\documentclass[11pt]{amsart}
\usepackage[nobysame,abbrev]{amsrefs}
\usepackage{amssymb}
\usepackage[all]{xy}

\DeclareFontFamily{U}{wncy}{}
\DeclareFontShape{U}{wncy}{m}{n}{<->wncyr10}{}
\DeclareSymbolFont{mcy}{U}{wncy}{m}{n}
\DeclareMathSymbol{\Sha}{\mathord}{mcy}{"58}
\DeclareMathSymbol{\sha}{\mathord}{mcy}{"78}

\newtheorem{thm}{Theorem}[section]
\newtheorem{cor}[thm]{Corollary}
\newtheorem{lem}[thm]{Lemma}
\newtheorem{prop}[thm]{Proposition}
\newtheorem{defin}[thm]{Definition}
\newtheorem{exam}[thm]{Example}
\newtheorem{examples}[thm]{Examples}
\newtheorem{rem}[thm]{Remark}

\newtheorem{conj}[thm]{Conjecture}

\swapnumbers

\newtheorem*{acknowledgment}{Acknowledgment}

\numberwithin{equation}{section}

\newcommand{\F}{\mathbb{F}}
\newcommand{\K}{\mathbb{K}}

\newcommand{\Z}{\mathbb{Z}}
\newcommand{\calG}{\mathcal{G}}

\newcommand{\op}{{\mathrm{op}}}

\newcommand{\sV}{{\mathcal V}}
\newcommand{\sE}{{\mathcal E}}
\newcommand{\calL}{{\mathcal L}}

\newcommand{\calY}{{\mathcal Y}}
\newcommand{\calS}{{\mathcal S}}
\newcommand{\calB}{{\mathcal B}}

\begin{document}

\title[RAAGs and enhanced Koszul properties]{Right-angled Artin groups \\ and enhanced Koszul properties}

\author{A. Cassella \and C. Quadrelli}


\address{Department of Mathematics and Applications, University of Milano Bicocca, 20125 Milan, Italy EU}
\email{a.cassella@campus.unimib.it}
\email{claudio.quadrelli@unimib.it}
\date{\today}

\begin{abstract}
Let $\F$ be a finite field.
We prove that the cohomology algebra $H^\bullet(G_\Gamma,\F)$ with coefficients in $\F$
of a right-angled Artin group $G_\Gamma$ is a strongly Koszul algebra
for every finite graph $\Gamma$.
Moreover, $H^\bullet(G_\Gamma,\F)$ is a universally Koszul algebra if, and only if, the graph $\Gamma$
associated to the group $G_\Gamma$ has the diagonal property.
From this we obtain several new examples of pro-$p$ groups, for a prime number $p$,
whose continuous cochain cohomology algebra with coefficients in the field of $p$ elements
is strongly and universally (or strongly and non-universally) Koszul. This provides 
new support to a conjecture on Galois cohomology of maximal pro-$p$ Galois groups of fields
formulated by J.~Min\'a\v{c} et al.
\end{abstract}



\maketitle

\section{Introduction}
\label{sec:intro}

Right-angled Artin groups --- RAAGs for short --- are a combinatorial construction that has played 
a prominent role in geometric group theory in the last decades.
A RAAG is defined by a presentation where all relations are commutators of weight 2 of the generators,
which comes equipped with a combinatorial graph whose vertices are the generators, and two vertices 
are joined by an edge whenever they commute.
RAAGs may seem the most elementary class among Artin groups, yet such groups have surprising richness
an flexibility, and this led to some remarkable applications.
(For an overview on RAAGs we refer to \cite{raags}.)

In the present paper we investigate {\sl enhanced Koszul properties} for the cohomology of finitely generated RAAGs.
It is well known that the cohomology algebra $H^\bullet(G_\Gamma,\F)=\bigoplus_{n\geq0}H^n(G_\Gamma,\F)$
of a RAAG $G_\Gamma$ with associated graph $\Gamma$, with coefficients in a finite field $\F$ (considered as trivial $G_\Gamma$-module) and endowed with
the graded-commutative cup-product
\[
 H^r(G_\Gamma,\F)\otimes_{\F} H^s(G_\Gamma,\F)\overset{\cup}{\longrightarrow} H^{r+s}(G_\Gamma,\F),\qquad r,s\geq0,
\]
is isomorphic to the {\sl exterior Stanley-Reisner algebra} induced by $\Gamma$, and thus it is
a {\sl quadratic algebra}, i.e., a graded algebra which is generated by elements of degree 1,
and with homogeneous defining relations of degree 2.
By a result of R.~Fr\"oberg, the algebra $H^\bullet(G_\Gamma,\F)$ is also {\sl Koszul} (cf. \cite{froberg} and \cite{papa}).
The Koszul property for quadratic algebras was singled out by S.~Priddy in \cite{priddy},
and yields exceptionally nice behavior in terms of cohomology
(cf. Definition~\ref{defin:koszul} below and \cite{pp:quad}*{Ch.~2}).
Koszul property is very restrictive, still it arises in various areas of mathematics, such as representation
theory, algebraic geometry, combinatorics, and Galois theory.

Recently, some stronger versions of the Koszul property were introduced and investigated in commutative
algebra (see, e.g., \cites{conca:UK,HHR,cdr,ehh}) and extended to the non-commutative setting (cf. \cite{piont}), 
and finally considered in the context of Galois cohomology (cf. \cite{MPPT,cq:2rel}).
In particular, one has the notion of {\sl strongly Koszul} algebra and {\sl universally Koszul} algebra.
These two ``enhanced versions'' of Koszulity are independent to each other, and imply the ``simple'' Koszulity
(see \S~\ref{ssec:K} and \cite{MPPT}*{\S~2}).
Usually, checking whether a given quadratic algebra is Koszul is a rather hard problem.
Surprisingly, testing these enhanced versions of the Koszul property may be easier, even though they are more restrictive.

For RAAGs we prove the following.

\begin{thm}\label{thm:strong}
 Let $\Gamma$ be a finite combinatorial graph and $G_\Gamma$ the associated RAAG, and let $\F$ be a finite field.
The cohomology algebra $H^\bullet(G_\Gamma,\F)$ is strongly Koszul.
\end{thm}

RAAGs {\sl of elementary type} are the RAAGs which are constructible starting from free abelian groups
and taking direct products with $\Z$ and free products (see Definition~\ref{defin:ETRAAGs}).
Equivalently, $G_\Gamma$ is of elementary type if $\Gamma$ has the diagonal property --- i.e., $\Gamma$
does not contain squares or length-3 paths as full subgraphs (see \cite{wolk} and Proposition~\ref{thm:ETRAAGs}).
E.g., if $\Gamma$ is complete or a star, then $G_\Gamma$ is of elementary type.
This property characterizes those RAAGs whose cohomology is universally Koszul.

\begin{thm}\label{thm:univ}
 Let $\Gamma$ be a finite combinatorial graph and $G_\Gamma$ the associated RAAG, and let $\F$ be a finite field.
The cohomology algebra $H^\bullet(G_\Gamma,\F)$ is universally Koszul if, and only if, $G_\Gamma$
is of elementary type.
\end{thm}

Interestingly, Theorems~\ref{thm:strong} and \ref{thm:univ} provide plenty of examples of strongly Koszul
algebras which are not universally Koszul.

For a prime number $p$ and a graph $\Gamma$, let $\calG_\Gamma$ denote the
{\sl pro-$p$ completion} of the RAAG $G_\Gamma$, and let $\F_p$ denote the field with $p$ elements.
A result of K.~Lorensen states that the (continuous cochain) $\F_p$-cohomology algebra of $\calG_\Gamma$
coincides with the algebra $H^\bullet(G_\Gamma,\F_p)$ (cf. \cite{lorensen}).
Thus, Theorems~\ref{thm:strong} and \ref{thm:univ} yield several new examples of pro-$p$ groups with $\F_p$-cohomology
which is strongly and universally (or non-universally) Koszul,
in particular among {\sl generalized pro-$p$ RAAG}, a class of pro-$p$ groups introduced in \cite{QSV}
(see \S~\ref{ssec:prop}).

This has great relevance in the context of Galois theory.
Let $\K$ be a field containing a root of 1 of order $p$, and let $\calG_{\K}$ denote the {\sl maximal pro-$p$ Galois
group} of $\K$ --- i.e., $\calG_{\K}$ is the Galois group of the maximal pro-$p$-extension of $\K$.
In the last two decades, Koszulity has gained importance in Galois cohomology, thanks to the work of L.~Positselski,
especially in connection with the celebrated Bloch-Kato conjecture
(see, e.g., \cites{polivis,posi:koszul,posi:b}).
In particular, Positselski conjectured that the $\F_p$-cohomology algebra of a maximal pro-$p$ Galois group $\calG_{\K}$
is Koszul, and this was shown to be true in some relevant cases (cf. \cites{posi:conj,MPQT,cq:onerel}).
More recently, in \cite{MPPT} J.~Min\'a\v{c} et al. conjectured
that $\F_p$-cohomology algebra of a maximal pro-$p$ Galois group $\calG_{\K}$ is universally Koszul,
and proved this in some cases.

\begin{conj}{\cite{MPPT}*{Conj.~2}}\label{conj:univ}
 Let $\K$ be a field containing a root of 1 of order $p$, and suppose that $\calG_{\K}$ is finitely generated.
Then the $\F_p$-cohomology algebra of $\calG_{\K}$ is universally Koszul.
\end{conj}

If $\Gamma$ is a graph with the diagonal property, then it is well known that the pro-$p$ RAAG $\calG_\Gamma$ associated to $\Gamma$ occurs as the maximal pro-$p$ Galois group $\calG_{\K}$ for some field $\K$ containing a root of 1 of order $p$ (see Proposition~\ref{prop:Galois} below).
On the other hand, it was recently shown that if $\Gamma$ is a graph without the diagonal property, then it can not occur as the maximal pro-$p$ Galois group $\calG_{\K}$ for any field  $\K$ containing a root of 1 of order $p$
(see \cite{cq:bk}*{Thm.~5.6} and \cite{SZ}*{Thm.~1.2}).
Therefore, from the pro-$p$ version of Theorem~\ref{thm:univ} one deduces the following Galois-theoretic result.

\begin{cor}\label{thm:RAAGs ET Galois}
 A pro-$p$ RAAG $\calG_\Gamma$ has universally Koszul $\F_p$-cohomology if, and only if,
there exists a field $\K$ containing a root of 1 of order $p$ such that $\calG_\Gamma\simeq\calG_{\K}$.
\end{cor}

This settles positively Conjecture~\ref{conj:univ} for the class of maximal pro-$p$ Galois groups of fields which are pro-$p$ RAAGs.

\begin{small}
 \begin{acknowledgment} \rm
The authors are grateful to F.W.~Pasini, as this paper was inspired by the talk he delivered at
the ``Mathematical Salad'' seminars at the University of Milano-Bicocca, Italy, in Dec.~2018, and to the organizers of that 
talk.
Also, the authors wish to thank P.~Spiga for the discussions with him about graphs, and
I.~Snopce and M.~Vannacci, as the joint work with the second-named author on pro-$p$ RAAGs
was also source of inspiration for this paper.
\end{acknowledgment}
\end{small}


\section{Quadratic algebras and Koszul properties}
A graded algebra over a field $\F$ is a graded associative algebra $A_\bullet$ which decomposes as direct sum of vector spaces  
$\bigoplus_{n\in\Z}A_n$ such that $A_n\cdot A_m\subseteq A_{m+n}$.
Hereinafter every graded algebra $A_\bullet$ is assumed to satisfy the following conditions: $A_n=0$ for 
$n<0$, $A_1=\F$, $\dim(A_n)<\infty$; and $\F$ is always assumed to be a finite field.

For a subset $S$ of $A_\bullet$, $(S)$ denotes the two-sided ideal of $A_\bullet$ generated by $S$.
Moreover, $A_+$ denotes the augmentation ideal $\bigoplus_{n\geq1}A_n$ of $A_\bullet$.
Finally, if $S$ is a subset of a vector space $V$, then $\langle S\rangle$ denotes the vector subspace generated by $S$.


\subsection{Quadratic algebras and Koszul algebras}

For a vector space $V$ let $T_\bullet(V)=\bigoplus_{n\geq0}V^{\otimes n}$ denote the tensor algebra generated by $V$.
A graded algebra $A_\bullet$ is called {\sl quadratic} if there exists an isomorphism of graded algebras
$$A_\bullet\simeq T_\bullet(V)/(\Omega)$$ for some vector space $V$ and some subspace $\Omega\leq V\otimes V$.
We write $T_\bullet(V)/(\Omega)= Q(V,\Omega)$.

An ideal $I\trianglelefteq A_\bullet$ inherits the grading from $A_\bullet$, i.e., 
$$I=\bigoplus_{n\geq1}I_n, \qquad I_n=A_n\cap I.$$
In particular, $A_+=\bigoplus_{n\geq1}A_n$ is the augmentation ideal of $A_\bullet$.

\begin{defin}\label{defin:koszul}\rm
A quadratic algebra $A_\bullet$ is said to be Koszul if it admits a resolution
\[
 \xymatrix{\cdots\ar[r] & P(2)_\bullet\ar[r] & P(1)_\bullet\ar[r]  & P(0)_\bullet\ar[r]& \F}
\]
of right $A_\bullet$-modules (with trivial action on $\F$), where for each $i\geq0$, $P(i)_\bullet = \bigoplus_{n\geq0}P(i)_n$
is a free graded
$A_\bullet$-module such that $P(n)_n$ is finitely generated for all $n\geq0$.
\end{defin}

We will not need the formal definition of Koszul algebra for our investigation.
For further properties of Koszul algebras we direct the reader to \cite{pp:quad}*{Ch.~2}
and to \cite{MPQT}*{\S~2}.

\begin{exam}\label{ex:tensor}\rm 
Let $V$ be a finite-dimensional vector space.
The tensor algebra $T_\bullet(V)$, the exterior algebra $\Lambda_\bullet(V)$, and the quadratic algebra $Q(V,V^{\otimes2})$ (called the
 {\sl trivial} quadratic algebra) are Koszul (cf. \cite{lodval}*{Exam.~3.2.5}).
\end{exam}

Given two quadratic algebras $A_\bullet=Q(A_1,\Omega_A)$ and $B_\bullet=Q(B_1,\Omega_B)$, one has 
the following constructions (cf. \cite{MPQT}*{Exam.~2.5}).

 \begin{itemize}
 \item[(a)] The {\sl direct product} of $A_\bullet$ and $B_\bullet$ is the quadratic algebra
 $A_\bullet\sqcap B_\bullet=Q(A_1\oplus B_1,\Omega)$, 
 with $$\Omega=\langle\Omega_A\cup\Omega_B\cup (A_1\otimes B_1)\cup (B_1\otimes A_1)\rangle.$$
\item[(b)] The {\sl wedge product} (or skew-symmetric tensor product) of $A_\bullet$ and $B_\bullet$ is the quadratic algebra
$A_\bullet\wedge B_\bullet=Q(A_1\oplus B_1,\Omega)$,
with $\Omega=\langle\Omega_A\cup\Omega_B\cup\Omega_\wedge\rangle$, where
$$\Omega_\wedge=\langle ab+ba, a\in A_1,b\in B_1\rangle\subseteq A_1\otimes B_1\oplus B_1\otimes A_1.$$
\end{itemize}
If both $A_\bullet$ and $B_\bullet$ are Koszul, then also their direct product and wedge product
are Koszul (cf. \cite{pp:quad}*{\S~3.1}).


\subsection{Enhanced Koszul properties}\label{ssec:K}

Let $A_\bullet$ be a graded algebra. 
For two ideals $I,J$ of $A_\bullet$, the {\sl colon ideal} $I:J$ is the 
ideal
\[
 I: J=\{a\in A_\bullet\mid a\cdot J\subseteq I\}.
\]
In particular, if $I=(0)$, then one has
$$(0): J=\mathrm{Ann}(J)=\{a\in A_\bullet\mid a\cdot J=0\}.$$
Note that for every ideals $I,J$ of $A_\bullet$, one has $I\subseteq I:J$ and $\mathrm{Ann}(J)\subseteq I:J$.

Following \cite{MPPT}*{\S~2.2}, we state the definitions of the following three ``enhanced versions'' of the Koszul property:
strong Koszulity, universal Koszulity, and the PBW property.

\begin{defin}\label{defin:SK}\rm (Cf. \cite{MPPT}*{Def.~12}.)
 A quadratic algebra $A_\bullet$ is said to be {\sl strongly Koszul} if $A_1$ has a basis
$\mathcal{X}=\{u_1,\ldots,u_d\}$ such that for every subset $\mathcal{Y}=\{u_{i_1},\ldots,u_{i_k}\}$ of $\mathcal{X}$
and for every $r\in\{1,\ldots,k-1\}$ the colon ideal $(u_{i_1},\ldots,u_{i_{r-1}}):(u_{i_r})$
is generated by a subset of $\mathcal{X}$.
\end{defin}

(See \cites{HHR,cdr} for the original definition of the strong Koszulity property in commutative algebra.)

For a quadratic algebra $A_\bullet$, let
\[
 \calL(A_\bullet)=\{I\trianglelefteq A_\bullet\mid I=A_\bullet\cdot I_1\}
\]
denote the set of all ideals of $A_\bullet$ generated by a subset of $A_1$.
In particular, both the trivial ideal $(0)$ and the augmentation ideal $A_+$ belong to $\calL(A_\bullet)$.

\begin{defin}\label{prop:def UK}\rm (Cf. \cite{MPPT}*{Prop.~17}.)
A quadratic algebra $A_\bullet$ is said to be {\sl universally Koszul} if
for every ideal $I\in\calL(A_\bullet)$ and every $b\in A_1\smallsetminus I_1$ one has $I:(b)\in\calL(A_\bullet)$.
\end{defin}

(See \cites{conca:UK,ctv} for the original definition of the universal Koszulity property in commutative algebra.)

\begin{examples}\rm
\begin{itemize}
 \item[(a)] Set $A_\bullet=\F[a]$, i.e., $A_\bullet$ is the free graded algebra on the generator $a$.
 The augmentation ideal $A_+$ is the ideal $(a)$, and one has $\calL(A_\bullet)=\{(0),(a)\}$.
 Then $(0):(a)=\mathrm{Ann}(a)$, and therefore $(0):(a)=(0)$, which lies in $\calL(A_\bullet)$.
 Hence, $A_\bullet$ is both strongly and universally Koszul.
 \item[(b)] Set $A_\bullet=\F[a]/(a^2)$, i.e., $A_\bullet$ is the algebra generated by $a$ and concentrated in degree 0 and 1.
 One has $A_+=A_1=(a)$, and $\calL=\{(0),(a)\}$.
 Moreover, $(0):(a)=\mathrm{Ann}(a)=A_1\in\calL(A_\bullet)$.
Hence, $A_\bullet$ is both strongly and universally Koszul.
\end{itemize}
\end{examples}

One has the following two properties for universally Koszul algebras.

\begin{prop}\label{prop:directprod UK}
 Let $A_\bullet$ and $B_\bullet$ be quadratic algebras.
 Then the direct product $A_\bullet\sqcap B_\bullet$ is universally Koszul if, and only if, 
 both $A_\bullet$ and $B_\bullet$ are universally Koszul.
\end{prop}

\begin{proof}
Let assume that $A_\bullet$ is not universally Koszul.
Then there exists an ideal $I$ of $A_\bullet$, $I\in\calL(A_\bullet)$, and an element $\beta\in A_1$,
such that $J\notin\calL(A_\bullet)$, for $J$ the colon ideal $I:(b)$.

Set $C_\bullet=A_\bullet\sqcap B_\bullet$, and let $\tilde I$ be the extension of $I$ in $C_\bullet$.
Then $\tilde I=I$, as $A_+\cdot B_+=0$.
Let $\tilde J$ denote the colon ideal $\tilde I:(b)\trianglelefteq C_\bullet$.
Then $\tilde J=J\oplus B_+$, and there exists an element $c\in J_+\subseteq\tilde J_+$ such that $c\notin(J_1)$,
and thus $c\notin(\tilde J_1)=(J_1+B_1)$.

The opposite implication is \cite{MPPT}*{Prop.~28}.
\end{proof}

\begin{prop}{\cite{MPPT}*{Prop.~31}}\label{prop:wedgeprod UK}
 Let $A_\bullet$ be a quadratic universally Koszul algebra, and let $V$ be a vector space of finite dimension.
Then the wedge product $A_\bullet\wedge \Lambda_\bullet(V)$ is universally Koszul.
\end{prop}

Finally, one has also the notion of {\sl PBW generators} of a quadratic algebra,
--- introduced in \cite{pp:quad}*{Ch.~4} --- which generalizes the notion of G-quadratic commutative algebra (cf. \cite{conca0}).
A quadratic algebra $A_\bullet=Q(V,\Omega)$ is called a {\sl PBW quadratic algebra}
if it admits generators for which the non-commutative Gr\"obner basis of relations consists
of elements of degree two (see \cite{lodval}*{\S~4.3} and \cite{MPPT}*{Def.~8}).
Namely, consider the lexicographical order $\prec$ on the set of multi-indices of length $n$ --- i.e., 
$(i_1,\ldots,i_n)\prec(j_1,\ldots,j_n)$ if, and only if, there exists $1\leq k\leq n$ such that 
$i_1=j_1, i_2=j_2,\ldots, i_{k-1}=j_{k-1}$ and $i_k<j_k$ for $i_h,j_h\in \{1,\ldots,d\}$, where $d=\dim(V)$.

Let $\{v_1,\ldots,v_d\}$ be a basis of $V$.
Then there exists $\calS\subseteq \{1,\ldots,d\}^2$ such that the relations in $\Omega$ can be written in the form
\[
 v_{i_1}v_{i_2}=\sum_{\substack{(j_1,j_2)\prec(i_1,i_2)\\(j_1,j_2)\in\calS}}\alpha v_{j_1}v_{j_2}, \qquad (i_1,i_2)\notin\calS,\alpha\in\F
\]
(cf. \cite{pp:quad}*{Lemma~4.1.1}).

\begin{defin}\label{defin:PBW}\rm
Given $A_\bullet$ and $\calS$ as above, set $\calS^{(0)}=\{\varnothing\}$, $\calS^{(1)}=\{1,\ldots,d\}$, and
\[
 \calS^{(n)}=\{(i_1,\ldots,i_n)\mid (i_{h},i_{h+1})\in\calS, h=1,\ldots,n-1\} \qquad\text{for }n\geq2.
\]
The elements $v_1,\ldots,v_d$ of a basis of $V$ are called {\sl PBW generators} of $A_\bullet$ if the set of monomials
$\{v_{i_1}\cdots v_{i_n}\mid(i_1,\ldots,i_n)\in\calS^{(n)}\}$ is a basis of $A_n$ for every $n\geq0$.
Such a quadratic algebra is called a PBW algebra.
\end{defin}

These ``enhanced Koszulity'' properties are independent to each other, namely, none implies any other.
On the other hand, if a quadratic algebra has one of these properties, then it is Koszul.
Altogether, one has the following picture (cf. \cite{MPPT}*{\S~1.2}):
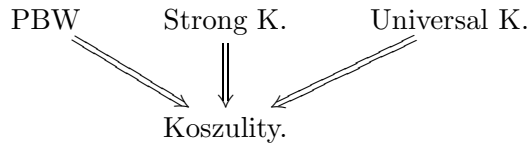
\[
 \xymatrix{ \text{PBW}\ar@{=>}[dr] & \text{Strong K.}\ar@{=>}[d] & \text{Universal K.}\ar@{=>}[dl] \\ & \text{Koszulity.} & } 
\]


\section{Right-angled Artin groups}

\subsection{Graphs}
For the notion of graph we refer to \cite{graph:book}*{Ch.~1}.
A na\"ive graph is a pair $\Gamma=(\sV,\sE)$ of sets where $\sE\subseteq[\sV]^2$, i.e., the elements
of $\sE$ are unordered subsets of 2 elements of $\sV$, which we shall denote by $(v,w)=(w,v)$, with $v,w\in\sV$.
The elements of $\sV$ are the vertices of the graph $\Gamma$, the elements of $\sE$ are its edges.
Moreover, we assume all graphs to have no loops, i.e., $(v,v)\notin\sE$ for any $v\in\sV$.
A graph $\Gamma=(\sV,\sE)$ is said to be finite if it has finite vertices.

Henceforth every graph will be assumed to be na\"ive and finite.
Here we list some definitions regarding graphs which will be used hereinafter.

\begin{defin}\label{defin:graph}\rm
Let $\Gamma=(\sV,\sE)$ be a graph.
\begin{itemize}
 \item[(i)] $\Gamma$ is a {\sl complete graph} if $\sE=[\sV]^2$.
 \item[(ii)]  A {\sl star graph} is a graph $\Gamma=(\sV,\sE)$ such that $\sV=\{w,v_1,\ldots,v_d\}$, with $d\geq2$,
 and $\sE=\{(w,v_1),\ldots,(w,v_d)\}$.
 \item[(iii)] A {\sl full subgraph} (or {\sl induced subgraph}) of $\Gamma$ is a subgraph $\Gamma'=(\sV',\sE')$
 of $\Gamma$ such that $\sE'=\sE\cap[\sV']^2$, i.e., if two vertices of $\sV'$ are joined by an edge of
$\Gamma$, then they are joined by an edge also in $\Gamma'$.
 \item[(iv)] For $n\geq1$, a {\sl $n$-clique} $\Gamma'$ of $\Gamma$ is a full subgraph $\Gamma'$ of $\Gamma$
with $n$ vertices which is a complete graph.
 \item[(v)] For $v,w\in\sV$, a {\sl path} from $v$ to $w$ is a subgraph $P=(\sV',\sE')$ of $\Gamma$ with 
 $\sV'=\{v_0=v,v_1,\ldots,v_{n-1},v_n=w\}$ and $\sE'=\{(v_0,v_1),(v_1,v_2),\ldots,(v_{n-1},v_n)\}$,
 and $n$ is the length of $P$. A path is a {\sl cycle} if $v=w$.
\end{itemize}
\end{defin}


\subsection{RAAGs and cohomology}\label{ssec:RAAGs}

Let $\Gamma=(\sV,\sE)$ be a graph, with vertices $\sV=\{v_1,\ldots,v_d\}$.
The {\sl right-angled Artin group} associated to $\Gamma$ is the group $G_\Gamma$ with presentation
\[
 G_\Gamma=\langle v_1,\ldots,v_d\mid [v_i,v_j]=1\text{ for }(v_i,v_j)\in\sE\rangle.
\]

The following is a well known result on RAAGs.

\begin{lem}\label{lem:RAAGs freeprod}
 Let $\Gamma$ be a graph, and suppose $\Gamma$ decomposes in connected components $\Gamma_1,\ldots,\Gamma_r$.
Then the RAAG $G_\Gamma$ decomposes as free product $G_{\Gamma_1}\ast\cdots\ast G_{\Gamma_r}$.
\end{lem}


\begin{defin}\label{defin:lambda gamma}\rm
Let $\Gamma=(\sV,\sE)$ be a graph, with $\sV=\{v_1,\ldots,v_d\}$, and let $V$ be the $\F$-vector space generated by $\sV^{\op}=\{a_1,\ldots,a_d\}$.
The {\sl exterior Stanley-Reisner algebra} $\Lambda_\bullet(\Gamma^{\op})$ over $\F$ associated to $\Gamma$
is the quotient of the exterior algebra $\Lambda_\bullet(V)$ over the two-sided ideal generated by 
\[
 \Omega=\langle a_i\wedge a_j \text{ for } (v_i,v_j)\notin \sE,1\leq i,j\leq d \rangle\subseteq \Lambda_2(V).
\]
\end{defin}

Since $\Lambda_\bullet(V)$ is quadratic and $\Omega\subseteq \Lambda_2(V)$, the algebra $\Lambda_\bullet(\Gamma^{\op})$ is quadratic.
While working with the algebra $\Lambda_\bullet(\Gamma^{\op})$ we will omit the wedge product $\wedge$
to denote the product of two elements, and we will just write $ab$ for the product of two elements $a$ and $b$ of $\Lambda_\bullet(\Gamma^{\op})$ --- in particular, $a_ia_j$ will denote the image of $a_i\wedge a_j$ in $\Lambda_2(\Gamma^{\op})$.
The result of R.~Fr\"oberg \cite{froberg} implies that $\Lambda_\bullet(\Gamma^{\op})$ is Koszul for any graph $\Gamma$.

The following result describes the $\F$-cohomology algebra of a RAAG (see \cite{papa}*{\S~3.2} and \cite{weigel:koszul}*{\S~4.2.2}).

\begin{prop}\label{prop:RAAA cohomology}
Let $G_\Gamma$ be the RAAG associated to a graph $\Gamma=(\sV,\sE)$.
Then the $\F$-cohomology algebra $H^\bullet(G_\Gamma,\F)$ of $G_\Gamma$
is isomorphic to the algebra $\Lambda_\bullet(\Gamma^{\op})$.
In particular, $\Lambda_\bullet(\Gamma^{\op})$ is Koszul.
\end{prop}

Thus, if a graph $\Gamma$ decomposes into connected components $\Gamma_1,\ldots,\Gamma_r$,
then the $\F$-cohomology algebra of the RAAG $G_\Gamma$ decomposes as direct product of quadratic algebras
\begin{equation}\label{eq:cohom decomp}
 H^\bullet(G_\Gamma,\F)\simeq \Lambda_\bullet(\Gamma^{\op})\simeq
 \Lambda_\bullet(\Gamma_1^{\op})\sqcap\ldots\sqcap\Lambda_\bullet(\Gamma_r^{\op}).
\end{equation}


\subsection{RAAGs of elementary type}

By Lemma~\ref{lem:RAAGs freeprod}, the free product of two RAAGs $G_{\Gamma_1}\ast G_{\Gamma_2}$ is the RAAG
with the disjoint union of $\Gamma_1$ and $\Gamma_2$ as associated graph.
Analogously, the direct product $G_\Gamma\times \Z$ of a RAAG $G_\Gamma$ with $\Z$ is isomorphic to the RAAG $G_{\tilde\Gamma}$
where $\tilde\Gamma $ is the cone graph with basis $\Gamma$, i.e., 
$$\sV(\tilde\Gamma)=\sV(\Gamma\dot\cup\{w\})\quad\text{and}\quad\sE(\tilde\Gamma)=\sE(\Gamma)\dot\cup\{(w,v),v\in\sV(\Gamma)\}.$$

For the following definition we mimic the definition of {\sl elementary type pro-$p$ groups},
defined by I.~Efrat (cf. \cite{Ido:ET}*{\S~3})

\begin{defin}\label{defin:ETRAAGs}\rm
 The class of {\sl RAAGs of elementary type} is the minimal class $\mathcal{C}$ of finitely generated RAAGs such that
 \begin{itemize}
\item[(a)] $\Z$ (considered as RAAGs with associated graph a single vertex) belongs to $\mathcal{C}$;
\item[(b)] if $G_{\Gamma_1}$ and $G_{\Gamma_2}$ belong to $\mathcal{C}$, then also 
$G_{\Gamma_1}\ast G_{\Gamma_2}$ belongs to $\mathcal{C}$;
\item[(c)] if $G_\Gamma$ belongs to $\mathcal{C}$, then also $G_\Gamma\times\Z$ belongs to $\mathcal{C}$.
\end{itemize}
\end{defin}

In other words, RAAGs of elementary type are precisely the RAAGs whose associated graphs are constructible starting
from the graph with a single vertex, via the following operations: disjoint union of graphs;
and cones.

Let $C_4$ and $P_4$ denote the cycle of length 4 and the path (non-cycle) of length 3 respectively, namely,
 \begin{equation}\label{eq:C P}
  \begin{minipage}{0.1\textwidth}
    \begin{center}$C_4$ = \end{center}
   \end{minipage}
   \begin{minipage}{0.3\textwidth}
    \xymatrix{  v_1\ar@{-}[d]\ar@{-}[r] & v_4\ar@{-}[d] \\ v_2\ar@{-}[r] & v_3 }
   \end{minipage}
   \begin{minipage}{0.1\textwidth}
    \begin{center} and \end{center}
   \end{minipage}
   \begin{minipage}{0.1\textwidth}
    \begin{center}$P_4$ = \end{center}
   \end{minipage}
  \begin{minipage}{0.3\textwidth}
    \xymatrix{v_1\ar@{-}[d] & v_4\ar@{-}[d] \\ v_2\ar@{-}[r] & v_3 }
   \end{minipage}
\end{equation}
A graph $\Gamma$ is said to have the {\sl diagonal property} if it does not contain a full subgraph isomorphic
to $C_4$ or $P_4$.
By the work of E.S.~Wolk \cite{wolk}, one has the following characterization of RAAGs of elementary type.

\begin{prop}\label{thm:ETRAAGs}
 Let $G_\Gamma$ be a RAAG with associated graph $\Gamma$.
 Then $G_\Gamma$ is of elementary type if, and only if, $\Gamma$ has the diagonal property.
\end{prop}

Moreover, by \cite{Droms}, every subgroup of a RAAG $G_\Gamma$ is again a RAAG if, and only if, 
$\Gamma$ has the diagonal property. 
In particular, every subgroup of a RAAG of elementary type is again a RAAG of elementary type.

\begin{exam}\label{exam:ETRAAGs}\rm
\begin{itemize}
 \item[(a)] All graphs with at most 3 vertices have the diagonal property, and thus yield a RAAG of elementary type.
 \item[(b)] A star graph $\Gamma$ is the cone graph with basis a disjoint union of vertices, and thus
 $G_\Gamma\simeq (\Z\ast\cdots\ast\Z)\times\Z$ is a RAAG of elementary type.
 \item[(c)] A complete graph $\Gamma$ may be obtained as iterated cone starting from a single vertex, indeed
 $G_\Gamma\simeq\Z\times\ldots\times\Z$ is a RAAG of elementary type.
\end{itemize}
\end{exam}


\section{RAAGs and enhanced Koszul properties}\label{ssec:RAAGs K}

Let $\Gamma=(\sV,\sE)$ be a graph, with $\sV=\{v_1,\ldots,v_d\}$, and set $\sV^{\op}=\{a_1,\ldots,a_d\}$.
Since the product in $\Lambda_\bullet(\Gamma^{\op})$ is graded-commutative, every ideal in $\Lambda_\bullet(\Gamma^{\op})$
is two-sided.

Given indices $1\leq i_1<\ldots<i_n\leq d$, one has $a_{i_1}\cdots a_{i_n}\neq0$ if, and only if,
there is a $n$-clique $\Gamma'$ of $\Gamma$ such that $\sV(\Gamma')=\{v_{i_1},\ldots,v_{i_n}\}$.
In particular, $\Lambda_n(\Gamma^{\op})=0$ if there are no $n$-cliques in $\Gamma$ --- which is always the case
if $n>d$.

Thus, for every $n\geq1$ the set 
\[\mathcal{B}_n=\{a_{i_1}\cdots a_{i_n}\mid 1\leq i_1<\ldots<i_n\leq d\text{ and }a_{i_1}\cdots a_{i_n}\neq0\}\]
is in 1-to-1 correspondence with the set of all $n$-cliques of $\Gamma$, and it is a basis of $\Lambda_n(\Gamma^{\op})$.
Moreover, for $\Delta=a_{i_1}\cdots a_{i_n}\in\calB_n$, we define $\sV(\Delta)=\{a_{i_1},\ldots,a_{i_n}\}\subseteq\sV^{\op}$
--- namely, $\sV(\Delta)$ corresponds to the vertices of the $n$-clique of $\Gamma$ associated to $\Delta$.

\begin{lem}\label{lem:cliques}
For a graph $\Gamma$, let $\calS$ be a subset of $\sV^{\op}$ and let $I\trianglelefteq\Lambda_\bullet(\Gamma^{\op})$
be the ideal generated by $\calS$.
Then $I_n$ is the subspace of $\Lambda_n(\Gamma^{\op})$ generated by
$$\calB(\calS)_n=\{\Delta\in\calB_n\mid\sV(\Delta)\cap\calS\neq\varnothing\}.$$
\end{lem}

\begin{proof}
We proceed by induction on $n$.
If $n=1$ then $\calB(\calS)_n=\calS$. 
If $n\geq2$, then by induction one has
\[\begin{split} 
    I_n &= I_1\cdot \Lambda_{n-1}(\Gamma^{\op})+\ldots+I_{n-1}\cdot \Lambda_1(\Gamma^{\op}) \\
    &=\langle\calS\rangle\cdot\langle\calB_{n-1}\rangle+\ldots+\langle\calB(\calS)_{n-1}\rangle\cdot\langle\calB_1\rangle
  \end{split}\]
and thus $I_n=\langle\calB(\calS)_n\rangle$.
\end{proof}

From the above description of $\Lambda_\bullet(\Gamma^{\op})$ one deduces easily the following.

\begin{cor}\label{thm_PBW}
Let $\Gamma=(\sV,\sE)$ be a graph. 
The exterior Stanley-Reisner algebra $A_\bullet=\Lambda_\bullet(\Gamma^{\mathrm{op}})$ is a PBW algebra.
\end{cor}

\begin{proof}
Set $\calS^{(n)}=\{(i_1,\ldots,i_n)\mid1\leq i_1<\ldots<i_n\leq d\}$ for every $n\geq1$.
Then one may write the relations of $A_\bullet$ as $a_ja_i=\alpha a_ia_j$, with $i<j$ and $\alpha=-1$ if $(v_i,v_j)\in\sE$
and $\alpha=0$ otherwise, and moreover $a_i^2=0$ for all $i$.
So, the sets $\calS^{(n)}$ are as in Definition~\ref{defin:PBW}.
Thus, the set $\sV^{\mathrm{op}}=\{a_1,\ldots,a_d\}$ is a set of PBW-generators of $A_\bullet$,
as $\{v_{i_1}\cdots v_{i_n}\mid(i_1,\ldots,i_n)\in\calS^{(n)}\}=\calB_n$ --- which is a basis of $A_n$ ---
for every $n\geq0$.
\end{proof}

\subsection{Strong Koszulity}

The first result we get is the strong Koszulity of $\Lambda_\bullet(\Gamma^{\op})$, regardless of the graph $\Gamma$.

\begin{thm}\label{thm:strong testo}
The exterior Stanley-Reisner algebra $\Lambda_\bullet(\Gamma^{\op})$ is strongly Koszul for any graph $\Gamma=(\sV,\sE)$.
\end{thm}

\begin{proof}
Set $A_\bullet=\Lambda_\bullet(\Gamma^{\op})$.
Our goal is to show that the basis $\sV^{\op}$ is the suitable basis of $A_1$ fulfilling the condition
as in Definition~\ref{defin:SK}.

Fix a subset $\calY=\{a_{i_1},\ldots,a_{i_n}\}$ of $\sV^{\op}$.
For $1\leq r\leq n$ set $\calS'=\{a_{i_1},\ldots,a_{i_{r-1}}\}$,
$I=(a_{i_1},\ldots,a_{i_{r-1}})=(\calS')$ and
\[
 J=I:(a_{i_r})=\{b\in A_\bullet\mid a_{i_r}\cdot b\in(a_{i_1},\ldots,a_{i_{r-1}})\}.
\]
Moreover, set  $\calS''=\{a_j\in\sV^{\op}\mid a_ja_{i_r}=0\}$ and 
\[ \calS=\calS'\cup\calS'' =\sV^{\op}\smallsetminus\left\{a_j\mid j\notin\{i_1,\ldots,i_{r-1}\}\text{ and }a_ja_{i_r}\neq0\right\}.
\]
Note that $a_{i_r}\in\calS$, as $a_{i_r}^2=0$.
In particular, $a_{i_r}a_j\in I$ if, and only if, $a_j\in \calS$, as $\calB_2$ is a basis of $A_2$.
Hence,  $\calS\subseteq J_1$, and $I\subseteq (\calS)\subseteq J$.
We claim that the ideals $(\calS)$ and $J$ coincide.

Let $b\in A_\bullet$ be such that $b\notin (\calS)$.
Thus, one may write
\begin{equation}\label{eq:Gamma SK}
 b=\alpha_1\Delta_1+\ldots+\alpha_m\Delta_m,\qquad \alpha_h\in\F^\times,
\end{equation}
where $\Delta_h\in\calB_{n_h}$ for each $h\in\{1,\ldots,m\}$.
By Lemma~\ref{lem:cliques}, $\Delta_h\in (\calS)$ --- respectively $\Delta\in (\calS')=I$,
$\Delta\in (\calS'')$ --- if, and only if, the intersection of $\sV(\Delta)$ with $\calS$ ---
respectively with $\calS'$ and with  $\calS''$ --- is not empty.
Since $b\notin(\calS)$, one has $\Delta_h\notin\calB(\calS)_{n_h}$ for some $h$ in \eqref{eq:Gamma SK},
and in this case $a_{i_r}\cdot \Delta_h\neq0$.
Therefore, one obtains
\begin{equation}\label{eq:SK long sum}
 \begin{split}   a_{i_r}b &= a_{i_r}\cdot \sum_{\Delta_h\in (\calS)}\alpha_h\Delta_h+a_{i_r}\cdot\sum_{\Delta_h\notin(\calS)}\alpha_h\Delta_h\\
   &= \sum_{\Delta_h\in I}\alpha_h\cdot a_{i_r}\Delta_h+\sum_{\Delta_h\notin(\calS)}\alpha_h\cdot a_{i_r}\Delta_h,
  \end{split}
\end{equation}
where $a_{i_r}\Delta_h\in\calB_{n_h+1}\smallsetminus\calB(\calS)_{n_h+1}$ and $a_{i_r}\Delta_h\neq a_{i_r}\Delta_{h'}$
for every $\Delta_h,\Delta_{h'}\notin(\calS)$, $h'\neq h$.
(Note that if $\Delta_h\in\calB(\calS'')_{n_h}$ then $a_{i_r}\Delta_h=0$.)
Consequently, the right-side summand of \eqref{eq:SK long sum} is not trivial, and by Lemma~\ref{lem:cliques}
it does not lie in $I$, whereas the left-side summand lies in $I$, so that $a_{i_r}b\notin I$.
Therefore, $b\notin J$, and this proves the inclusion $J\subseteq(\calS)$.
\end{proof}

Theorem~\ref{thm:strong} follows from Theorem~\ref{thm:strong testo}
together with Proposition~\ref{prop:RAAA cohomology}.

\begin{rem}\rm
 Theorem~\ref{thm:strong testo} provides a new proof of the fact that the algebra $\Lambda_\bullet(\Gamma^{\mathrm{op}})$
 is Koszul.
\end{rem}

\subsection{Universal Koszulity}
The following is a direct consequence of Proposition~\ref{prop:directprod UK} and of \eqref{eq:cohom decomp}.

\begin{prop}\label{prop:freeprod UK}
 Let $\Gamma$ be the disjoint union of two graphs $\Gamma_1$ and $\Gamma_2$.
Then $\Lambda_\bullet(\Gamma^{\mathrm{op}})$ is universally Koszul if, and only if, both  $\Lambda_\bullet(\Gamma_1^{\mathrm{op}})$
and $\Lambda_\bullet(\Gamma_2^{\mathrm{op}})$ are universally Koszul.
\end{prop}

\begin{thm}\label{thm:RAAGs UK}
Let $\Gamma=(\sV,\sE)$ be a graph.
The exterior Stanley-Reisner algebra $\Lambda_\bullet(\Gamma^{\op})$ is universally Koszul
if, and only if, $\Gamma$ has the diagonal property.
\end{thm}

\begin{proof}
Set $A_\bullet=\Lambda_\bullet(\Gamma^{\mathrm{op}})$ and $\sV=\{v_1,\ldots,v_d\}$.

Suppose first that $\Gamma$ has the diagonal property.
We proceed by induction on $d$.
If $d=1$ then $A_\bullet\simeq \F[a]/(a^2)$, which is universally Koszul.
If $d=2$ then either $A_\bullet$ is the exterior algebra generated by $V$, or it is the trivial algebra generated
by $V$, with $V$ a space of dimension 2, so that it is universally Koszul.

If $d\geq3$, then either $\Gamma$ decomposes as disjoint union of two proper full subgraphs $\Gamma_1$ and $\Gamma_2$,
or it is the cone graph with basis a full subgraph $\tilde\Gamma$.
In the former case, both $\Lambda_\bullet(\Gamma_1^{\mathrm{op}})$ and $\Lambda_\bullet(\Gamma_2^{\mathrm{op}})$ are 
universally Koszul by induction, so that also
$$A_\bullet\simeq\Lambda_\bullet(\Gamma_1^{\mathrm{op}})\sqcap\Lambda_\bullet(\Gamma_2^{\mathrm{op}})$$
is universally Koszul by Proposition~\ref{prop:freeprod UK};
in the latter case one has
$$A_\bullet\simeq\Lambda_\bullet(\tilde\Gamma^{\mathrm{op}})\wedge B_\bullet,\qquad B_\bullet\simeq\F[b]/(b^2),$$
and $\Lambda_\bullet(\tilde\Gamma^{\mathrm{op}})$ is universally Koszul by induction, so that also $A_\bullet$ is 
universally Koszul by Proposition~\ref{prop:wedgeprod UK}.

Suppose now that $\Gamma$ does not have the diagonal property.
Thus, $\Gamma$ contains a full subgraph $\Gamma'$ isomorphic to $C_4$ or $P_4$.
Set $\sV=\{v_1,\ldots,v_d\}$, with $d\geq4$, such that $\sV(\Gamma')=\{v_1,v_2,v_3,v_4\}$,
with the vertices labelled as in \eqref{eq:C P}.
Also, set $\sV^{\op}=\{a_1,\ldots,a_d\}$ and $A_\bullet=\Lambda_\bullet(\Gamma^{\op})$.

Set $b=a_1+a_4$, and set $J=(0):(b)=\mathrm{Ann}(b)$.
Then $$b\cdot a_2a_3=a_1a_2a_3+a_4a_2a_3=0,$$ and $a_2a_3\in J_2$.
We claim that $a_2a_3\notin J_1\cdot A_1$, i.e, $a_2a_3$ does not lie in the ideal generated by $J_1$,
so that $J\notin \calL(A_\bullet)$.
Clearly, $a_2,a_3\notin J_1$, as $ba_2=a_1a_2\neq0$ and $a_3b=a_3a_4\neq0$.
Suppose there exist $c_1,\ldots,c_r\in J_1$ and $c_1',\ldots,c_r'\in A_1$ such that $c_1c'_1+\ldots+c_rc_r'=a_2a_3$,
and write 
\[
 c_h=\alpha_{1,h}a_1+\ldots+\alpha_{d,h}a_d\qquad\text{and}\qquad c'_h=\beta_{1,h}a_1+\ldots+\beta_{d,h}a_d
\]
for each $h\in\{1,\ldots,r\}$, with $\alpha_{i,h},\beta_{i,h}\in\F$.
Since $c_h\in \mathrm{Ann}(b)$ for every $h$, one has $\alpha_{2,h}=\alpha_{3,h}=0$,
otherwise $bc_h=\alpha_{2,h}a_1a_2-\alpha_{3,h}a_3a_4+\Delta$, with $\Delta\in A_2$ a combination of elements $a_ia_j$
with $i<j$, $(1,2),(3,4)\neq(i,j)$, and $bc_h\neq0$ as $\calB_2=\{a_ia_j\mid 1\leq i<j\leq d,a_ia_j\neq0\}$
is a basis of $A_2$.
Thus, one obtains
\[\begin{split}
    \sum_{h=1}^rc_hc_h' &= \sum_{i<j}a_ia_j\left(\sum_{h=1}^r\alpha_{i,h}\beta_{j,h}-\sum_{h=1}^r\alpha_{j,h}\beta_{i,h}\right)\\
    &= a_2a_3\left(\sum_{h=1}^r\alpha_{2,h}\beta_{3,h}-\sum_{h=1}^r\alpha_{3,h}\beta_{2,h}\right)+\Delta \\ &=0+\Delta,
  \end{split}
\]
with $\Delta\in A_2$ a combination of elements $a_ia_j$, $(i,j)\neq(2,3)$, a contradiction.

Therefore, $a_2a_3\in J_2\smallsetminus J_1\cdot A_1$, and consequently $J\notin\calL(A_\bullet)$, 
and $A_\bullet$ is not universally Koszul.
\end{proof}

Theorem~\ref{thm:univ} follows from Theorem~\ref{thm:RAAGs UK},
together with Proposition~\ref{prop:RAAA cohomology}, equation \eqref{eq:cohom decomp} and
Proposition~\ref{prop:freeprod UK}.

In the following two examples we show explicitly that $\Lambda_\bullet(\Gamma^{\mathrm{op}})$ is universally
Koszul for two graphs with the diagonal property, without using Proposition~\ref{prop:wedgeprod UK}.

\begin{exam}\label{prop:starUK}\rm
 Let $\Gamma=(\sV,\sE)$ be a star graph with $\sV=\{w,v_1,\ldots,v_d\}$ and $\sE=\{(w,v_1),\ldots,(w,v_d)\}$.
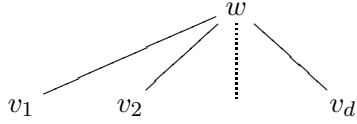
 \[
  \xymatrix{ & & w\ar@{-}[dll]\ar@{-}[dl]\ar@{..}[d]\ar@{-}[dr] & \\ v_1 & v_2 &   & v_d}
 \]
Set $\sV^{\op}=\{a_w,a_1,\ldots,a_d\}$ and $A_\bullet=\Lambda_\bullet(\Gamma^{\op})$.
Then
\[\begin{split}
   A_1 &= \langle a_w,a_1,\ldots,a_d\rangle, \\ A_2 &= \langle a_w a_1,\ldots,a_w a_d\rangle= a_w\wedge A_1,\\
   A_n &= 0\quad\text{for }n\geq3.
  \end{split}\]
In particular, $a_ia_j=0$ for all $i,j$.
 For $I\in\calL(A_\bullet)$ and $b\in A_1\smallsetminus I_1$, write $b=\alpha_w a_w+\sum_i \alpha_ia_i$,
with $\alpha_w,\alpha_i\in\F$.
Set $J=I:(b)$. 
Then $b\in J$, and moreover $J_2=A_2$, as $A_2\cdot b\subseteq A_3=0$. 
If $\alpha_w\neq0$, then $A_1\cdot b= A_1\cdot\alpha_w a_w=A_2=J_2$,
so that $J$ is 1-generated, i.e., $J\in\calL(A_\bullet)$.
If $\alpha_w=0$, then $a_ib=0$ for all $i\in\{1,\ldots,d\}$, and thus $a_i\in J_1$.
In this case, $a_wa_i\in A_1\cdot J_1\subseteq J_2$ for all $i$, and therefore $A_1\cdot J_1=A_2$,
so that $J$ is 1-generated, i.e., $J\in\calL(A_\bullet)$.
Thus, $A_\bullet$ is universally Koszul.
\end{exam}

\begin{exam}\label{prop:Gamma4 2}\rm
Let $\Gamma$ be the graph 
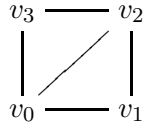
\[	\xymatrix{  v_3\ar@{-}[d]\ar@{-}[r] & v_2\ar@{-}[d] \\ v_0\ar@{-}[r]\ar@{-}[ru] & v_1 }	\]
The graph $\Gamma$ is the cone with vertex $v_0$ and basis the full subgraph with vertices $v_1,v_2,v_3$.
Set $\sV^{\mathrm{op}}=\{a_0,a_1,a_2,a_3\}$, with $a_i$ dual to $v_i$ for all $i$,
and $A_\bullet=\Lambda_\bullet(\Gamma^{\mathrm{op}})$.
Then 
\[\begin{split}
   A_1 &= \langle a_0,a_1,a_2,a_3\rangle, \\ 
   A_2 &=\langle a_0a_1,a_0a_2,a_0a_3,a_1a_2,a_2a_3\rangle=A_1\wedge\langle a_0,a_2\rangle,\\
   A_3 &=\langle a_0a_1a_2,a_0a_2a_3\rangle,\\
   A_n &= 0\quad\text{for }n\geq3.
  \end{split}\]
Set $$b=\alpha_0a_0+\alpha_1a_1+\alpha_2a_2+\alpha_3a_3\neq0.$$
Without loss of generality we can suppose that either $b=a_2$ (if $\mathrm{Ann}(b)_1$ has dimension 1)
or $b=a_3$ (if $\mathrm{Ann}(b)_1$ has dimension 2), by performing a change of basis for $A_1$ inducing
an automorphism of $A_\bullet$.
For $I=\calL(A_\bullet)$ such that $b\notin I_1$, set $J=I:(b)$, and $K=\mathrm{Ann}(b)\subseteq J$.
In both cases one has $K_3=J_3=A_3$, whereas $K_1 =\langle b\rangle$ and
$K_2=\langle ba_0, ba_1, ba_3\rangle$, for $b=a_2$, and $K_1 =\langle a_1,a_3\rangle$ and
$K_2=\langle a_0a_1,a_0a_3,a_1a_2,a_2a_3\rangle$, for $b=a_3$.
Moreover, $J_1=K_1+I_1$.
In particular, $K\in\calL(A_\bullet)$ in both cases, so that if $I=(0)$, then $J=K\in\calL(A_\bullet)$.

Suppose now that $I\neq(0)$.
Then either $\dim(I_3)=1$ or $\dim(I_3)=2$.
In particular, $\dim(I_3)=1$ if, and only if, $I=(b')$ for some $b'\in\langle a_1,a_3\rangle$;
whereas $\dim(I_3)=2$ if, and only if, either $I=(b')$ for some $b'\notin\langle a_1,a_3\rangle$, or $\dim(I_1)\geq2$.
Altogether, one has the following cases:
\begin{itemize}
 \item[(a)] Suppose $\dim(I_3)=1$. If $b=a_3$, we may assume without loss of generality that $b'=a_1$, so that $J_2=K_2$.
 If $b=a_2$ then $J_2=K_2+\langle a_0b'\rangle\subseteq(J_1)$.
 \item[(b)] Suppose $I_3=A_3$, so that $J_2=A_2$.
 If $b=a_3$ then one may find $b''=\gamma_0a_0+\gamma_1a_1+\gamma_2a_2\in J_1=K_1+I_1$, with $\gamma_i\in \F$,
 $(\gamma_0,\gamma_2)\neq(0,0)$, so that $a_0a_2\in (b'')_2+K_2$.
 Therefore $A_2=K_2\oplus\langle a_0a_2\rangle\subseteq (J_1)$.
 
 \noindent If $b=a_2$, then one has two further subcases.
 If $\dim(I_1)=1$, then one may find $b''=a_0+\gamma_1a_1+\gamma_3a_3\in J_1=K_1+I_1$, with $\gamma_1,\gamma_3\in\F$,
 as $b'\notin\langle a_1,a_3\rangle$ and $b=a_2,b'\in J_1$.
 Thus $A_2=K_2\oplus\langle a_0a_1,a_0a_3\rangle\subseteq(J_1)$.
 On the other hand, if $\dim(I_1)\geq2$ and $J_1$ does not contain an element $b''$ as above,
 then necessarily $a_1,a_3\in J_1$; in both cases $a_0a_1,a_0a_3\in(J_1)$ and $A_2=K_2\oplus\langle a_0a_1,a_0a_3\rangle\subseteq(J_1)$.
\end{itemize}
In every case one has $J\in\calL(A_\bullet)$.
\end{exam}

\section{Pro-$p$ groups}


\subsection{Right-angled Artin pro-$p$ groups}\label{ssec:prop}

Fix a prime number $p$, and let $\F_p$ be the finite field with $p$ elements.
For a pro-$p$ group $\calG$, we consider $\F_p$ as continuous trivial $\calG$-module.
The {\sl continuous cochain} cohomology algebra 
\[H^\bullet(\calG,\F_p)=\bigoplus_{n\geq0} H^n(\calG,\F_p)\]
of $\calG$ with coefficients in $\F_p$,
endowed with the graded-commutative cup-product, is a graded $\F_p$-algebra.
In particular, one has isomorphisms of $p$-elementary abelian groups
\[
 H^1(\calG,\F_p)\simeq\mathrm{Hom}(\calG/\Phi(\calG),\F_p),
\]
where $\Phi(\calG)$ denotes the Frattini subgroup of $\calG$, i.e., the closed subgroup generated by all
$g^p$ and $[g,h]$, with $g,h\in\calG$.
For the definition and the properties of continuous cochain cohomology of pro-$p$ groups
we refer to \cite{nsw:cohn}*{\S~I.2--I.4}.

Given a graph $\Gamma$ we call the pro-$p$ completion $\calG_\Gamma$ of the RAAG $G_\Gamma$
a {\sl pro-p RAAG} with associated graph $\Gamma$.
Pro-$p$ RAAG behave pretty much like abstract RAAGs (see, e.g., \cite{KW}).
In particular, one has the following result by K.~Lorensen (cf. \cite{lorensen}*{Thm.~2.6}).

\begin{thm}\label{thm:prop cohomology}
Let $\Gamma$ be a graph. Then $H^\bullet(\calG_\Gamma,\F_p)\simeq H^\bullet(G_\Gamma,\F_p)$. 
\end{thm}

Moreover, one has the ``pro-$p$ equivalent'' of Lemma~\ref{lem:RAAGs freeprod} and of \eqref{eq:cohom decomp},
namely, if a graph $\Gamma$ decomposes into connected components $\Gamma_1,\ldots,\Gamma_r$,
then the pro-$p$ RAAG $\calG_\Gamma$ decomposes as free pro-$p$ product
$\calG_\Gamma\simeq \calG_{\Gamma_1}\ast_{\hat p}\cdots\ast_{\hat p}\calG_{\Gamma_r}$, where $\ast_{\hat p}$
denotes the free product in the category of pro-$p$ groups (see \cite{ribzal}*{\S~9.1}), and
the $\F_p$-cohomology algebra of $\calG_\Gamma$ decomposes as direct product of quadratic algebras
\begin{equation}\label{eq:cohom decomp p}
 H^\bullet(\calG_\Gamma,\F_p)\simeq \Lambda_\bullet(\Gamma^{\op})\simeq
 \Lambda_\bullet(\Gamma_1^{\op})\sqcap\ldots\sqcap\Lambda_\bullet(\Gamma_r^{\op}).
\end{equation}
Thus, one may extend Theorem~\ref{thm:strong} and Theorem~\ref{thm:univ} to the class of pro-$p$ RAAGs.

\begin{thm}\label{thm:prop RAAGs}
Let $\Gamma$ be a graph and $\calG_\Gamma$ the associated pro-$p$ RAAG.
\begin{itemize}
 \item[(i)] The $\F_p$-cohomology algebra $H^\bullet(\calG_\Gamma,\F)$ is strongly Koszul and PBW. 
 \item[(ii)] The $\F_p$-cohomology algebra $H^\bullet(\calG_\Gamma,\F)$ is universally Koszul if, and only if,
 $\Gamma$ has the diagonal property.
 \end{itemize} 
\end{thm}

A class of pro-$p$ groups which are very similar to pro-$p$ RAAGs is the class of {\sl generalized pro-$p$ RAAGs},
introduced and studied in \cite{QSV}.
Given a graph $\Gamma=(\sV,\sE)$, a generalized pro-$p$ RAAG with associated graph $\Gamma$ is
a pro-$p$ group $\calG$ generated by $\sV=\{v_1,\ldots,v_d\}$ and with defining relations $[v_i,v_j]v_i^\alpha v_i^\beta$
for some $\alpha,\beta\in p\Z_p$, if $(v_i,v_j)\in\sE$ (moreover, $\alpha,\beta\in4\Z_2$ if $p=2$).
Namely, one has a presentation (of pro-$p$ group)
\[
 \calG=\left\langle v_1,\ldots,v_d \left\vert [v_i,v_j]v_i^{\alpha_{ij}} v_i^{\beta_{ij}}=1,(v_i,v_j)\in\sE,\alpha_{ij},\beta_{ij}\in p^\epsilon\Z_p\right.\right\rangle_{\hat p},
\]
with $\epsilon=1$ if $p>2$, $\epsilon=2$ otherwise.

A priori, a generalized pro-$p$ RAAG may have $\F_p$-cohomology algebra which is not quadratic ---
e.g., a generalized pro-$p$ RAAG $\calG$ may be a finite group, so that $H^\bullet(\calG,\F_p)$ is not quadratic unless
$p=2$ and $\calG$ is 2-elementary abelian (cf. \cite{QSV}*{Ex.~5.16}).
Yet, if $H^\bullet(\calG,\F_p)$ is quadratic for a generalized pro-$p$ RAAG $\calG$ with associated graph $\Gamma$,
then one has  $H^\bullet(\calG,\F_p)\simeq\Lambda_\bullet(\Gamma^{\op})$,
just like (pro-$p$) RAAGs (cf. \cite{QSV}*{Thm.~E}).
Hence, by Theorem~\ref{thm:strong testo} the $\F_p$-cohomology algebra of a generalized pro-$p$ RAAG
$\calG$ is strongly Koszul, too.
In \cite{QSV}*{\S~5.5-5.6} uncountably many examples of generalized pro-$p$ RAAGs are shown
to have quadratic $\F_p$-cohomology algebra --- e.g., if $\Gamma$ contains no triangles as full subgraphs
(cf. \cite{QSV}*{Thm.~F}).
Thus, one may deduce the following.

\begin{cor}
 Let $\calG$ be a generalized pro-$p$ RAAG with associated graph $\Gamma$,
and suppose that $H^\bullet(\calG,\F_p)$ is a quadratic algebra.
Then $H^\bullet(\calG,\F_p)$ is a strongly Koszul and PBW algebra.
Moreover, it is universally Koszul if, and only if, $\Gamma$ has the diagonal property.
\end{cor}

\begin{exam}\label{exam:prop star}\rm
Let $\calG$ be the pro-$p$ group with presentation
\[
 \calG\simeq\left\langle w,v_1,\ldots,v_d\mid [v_i,w]=v_i^{\alpha_i},i=1,\ldots,d,\alpha_i\in p^\epsilon\Z_p \right\rangle_{\hat p},
\]
where $\epsilon=1$ if $p>2$, $\epsilon=2$ otherwise.
Then $\calG$ is a generalized pro-$p$ RAAG with associated graph a star graph $\Gamma$ with center $w$.
Thus, by \cite{QSV}*{Thm.~F} one has $H^\bullet(\calG,\F_p)\simeq\Lambda_\bullet(\Gamma^{\mathrm{op}})$,
which is strongly and universally Koszul, and PBW.
\end{exam}

\begin{exam}\label{exam:propraag1}\rm
Let $\calG$ be the pro-$p$ group with presentation
\[ \calG\simeq\left\langle v_0,\ldots,v_4\mid [v_0,v_i]=v_0^{\alpha_i}v_i^{\beta_i}, [v_1,v_2]=v_1^{\alpha'_1}v_2^{\beta'_1},
[v_2,v_3]=v_2^{\alpha'_3}v_3^{\beta'_3} \right\rangle_{\hat p}\]
 with $i=1,\ldots,4$, and $\alpha_i,\beta_i,\alpha_i',\beta_i'\in p^\epsilon\Z_p$, where $\epsilon=1$ if $p>2$, $\epsilon=2$ otherwise.
Then $\calG$ is a generalized pro-$p$ RAAG with associated graph
  \[\begin{minipage}{0.1\textwidth}
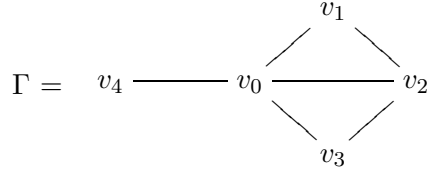
    \begin{center}$\Gamma$ = \end{center}
   \end{minipage}
   \begin{minipage}{0.3\textwidth}
    \xymatrix@=0.5truecm{   & & & v_1\ar@{-}[dr]\ar@{-}[dl] & \\ v_4\ar@{-}[rr] && v_0\ar@{-}[rr] & & v_2
    \\  & & & v_3\ar@{-}[ur]\ar@{-}[ul] &}
   \end{minipage}
\]
The graph $\Gamma$ is the cone graph with basis the full subgraph with vertices $v_1,\ldots,v_4$, and it has
the diagonal property.
By \cite{QSV}*{\S~5.6}, one has $H^\bullet(\calG,\F_p)\simeq\Lambda_\bullet(\Gamma^{\mathrm{op}})$,
which is strongly and universally Koszul, and PBW.
\end{exam}

\begin{exam}\label{exam:propraag2}\rm
Let $\calG$ be the pro-$p$ group with presentation
\[
 \calG\simeq\left\langle v_0,\ldots,v_4\mid [v_0,v_i]=v_0^{\alpha_i}v_i^{\alpha_i},[v_1,v_j]=v_1^{\alpha'}v_j^{\beta'},
 [v_4,v_j]=v_4^{\alpha''}v_j^{\beta''}\right\rangle_{\hat p}\]
with $i=1,4$, $j=2,3$,
and $\alpha_i,\beta_i,\alpha',\beta',\alpha'',\beta''\in p^\epsilon\Z_p$, where $\epsilon=1$ if $p>2$, $\epsilon=2$ otherwise.
Then $\calG$ is a generalized pro-$p$ RAAG with associated graph
  \[
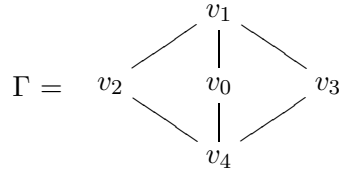
  \begin{minipage}{0.1\textwidth}
    \begin{center}$\Gamma$ = \end{center}
   \end{minipage}
   \begin{minipage}{0.3\textwidth}
    \xymatrix@R=0.5truecm{   & v_1\ar@{-}[d]\ar@{-}[rd]\ar@{-}[ld] &  \\ v_2\ar@{-}[rd] & v_0\ar@{-}[d] & v_3\ar@{-}[ld]
    \\  & v_4 &  }
   \end{minipage}
\]
The graph $\Gamma$ does not have the diagonal property.
By \cite{QSV}*{\S~5.6}, one has $H^\bullet(\calG,\F_p)\simeq\Lambda_\bullet(\Gamma^{\mathrm{op}})$,
which is strongly Koszul and PBW, but not universally Koszul.
\end{exam}

One may find a further class of pro-$p$ groups whose $\F_p$-cohomology algebras
are both strongly and universal Koszul.
A finitely generated pro-$p$ group is called {\sl uniform} if it is torsion-free and $\Phi(\calG)\subseteq\bar{\calG}^{p^\epsilon}$,
where $\bar\calG^{p^\epsilon}$ denotes the closed subgroup of $\calG$ generated by the $p^\epsilon$-powers
of elements of $\calG$, where $\epsilon=1$ if $p>2$ and $\epsilon=2$ otherwise (cf. \cite{ddsms}*{\S~4.1}).
For these pro-$p$ groups one has the following.

\begin{prop}\label{prop:uniform}
Let $\calG$ be a uniform pro-$p$ group.
Then the $\F_p$-cohomology algebra $H^\bullet(\calG,\F_p)$ is strongly and universally Koszul, and PBW.
\end{prop}

\begin{proof}
Set $V=H^1(\calG,\F_p)$.
Then by \cite{sw}*{Thm.~5.1.5} one has $H^\bullet(\calG,\F_p)\simeq\Lambda_\bullet(V)$.
The claim follows from Theorem~\ref{thm:strong testo} and Theorem~\ref{thm:RAAGs UK}.
\end{proof}


\subsection{Maximal pro-$p$ Galois groups}

In this subsection $\K$ will denote a field containing a root of 1 of order $p$.
By the positive answer to the {\sl Bloch-Kato conjecture} given by M.~Rost and V.~Voevodsky (cf. \cites{rost,voev,weibel}),
one knows that the maximal pro-$p$ Galois group $\calG_{\K}$ of $\K$ has quadratic $\F_p$-cohomology
algebra (see, e.g., \cite{cq:bk}*{\S~2}).
In \cite{MPPT} it is conjectured that $\calG_{\K}$ has universally Koszul $\F_p$-cohomology algebra
(cf. Conjecture~\ref{conj:univ}).

\begin{prop}\label{prop:Galois}
 If $\Gamma$ is a graph with the diagonal property,
then for any prime $p$ the pro-$p$ RAAG $\calG_\Gamma$ associated to $\Gamma$ occurs as the maximal pro-$p$ Galois group $\calG_{\K}$
for some field $\K$ containing a root of 1 of order $p$.
\end{prop}

\begin{proof}
If a finitely generated pro-$p$ group $\calG$ occurs as the maximal pro-$p$ Galois group $\calG_{\K}$
for some field $\K$ containing a root of 1 of order $p$,
then one has $\calG\times\Z_p\simeq\calG_{\K(\!(X)\!)}$, with $\K(\!(X)\!)$ the field of Laurent series
in one indeterminate $X$ with coefficients in $\K$. In particular, 
$\Z_p=\{1\}\times\Z_p\simeq \calG_{\K(\!(X)\!)}$ with $\K$ a $p$-closed field (i.e., $\calG_{\K}=\{1\}$).
On the other hand, if two finitely generated pro-$p$ groups $\calG_1,\calG_2$ occur as maximal pro-$p$ Galois groups,
then also their free product is realizable as the maximal pro-$p$ Galois group for some field $\K$ (cf. \cite{Ido:ET}*{Rem.~3.4}).

We proceed by induction on the number $d$ of vertices of $\Gamma$.
If $\Gamma$ has the diagonal property, then $\calG_\Gamma$ is constructible starting from free abelian pro-$p$ groups
by operating direct product with $\Z_p$ and free products.
Therefore, if $d=1$ then $\calG_\Gamma\simeq\Z\simeq\calG_{\K}$ for some suitable $\K$.
If $d\geq2$, then either $\calG_\Gamma\simeq\calG_{\Gamma_1}\ast_{\hat p}\calG_{\Gamma_2}$, with $\Gamma=\Gamma_1\dot\cup\Gamma_2$;
or $\calG_\Gamma\simeq\Z\times\calG_{\tilde\Gamma}$, with $\Gamma$ the cone graph with basis $\tilde\Gamma$
and the claim follows by the above argument.
\end{proof}

\begin{rem}\rm
 If $\Gamma$ is a graph with the diagonal property,
 then the associated pro-$p$ RAAG $\calG_\Gamma$ is a {\sl pro-$p$ group of elementary type}
(cf. \cite{Ido:ET}*{\S~3} and \cite{MPQT}*{\S~4}).
\end{rem}

On the other hand, it was recently shown that the only pro-$p$ RAAGs which may occur as maximal pro-$p$ Galois groups of fields containing a root of 1 of order $p$ are those associated to graphs with the diagonal property.
Indeed, let $C_4$ and $P_4$ be as in \eqref{eq:C P}.
The pro-$p$ RAAG $\calG_{C_4}$ associated to $C_4$ is isomorphic to the direct product
$\mathcal{F}\times\mathcal{F}$, with $\mathcal{F}$ a 2-generated free pro-$p$ group, and by \cite{cq:bk}*{Thm.~5.6}, $\calG_{C_4}$ --- and therefore also any other pro-$p$ RAAG $\calG_\Gamma$
with $\Gamma$ containing $C_4$ as full subgraph, --- is not realizable as the maximal pro-$p$ Galois group $\calG_{\K}$
for any $\K$.
Moreover, I.~Snopce and P.~Zalesskii recently proved that $\calG_{P_4}$ --- and therefore also any other pro-$p$ RAAG $\calG_\Gamma$
with $\Gamma$ containing $P_4$ as full subgraph, --- is not realizable as the maximal pro-$p$ Galois group $\calG_{\K}$
for any $\K$ (cf. \cite{SZ}*{Thm.~1.2}).

From this, together with Proposition~\ref{prop:Galois} and Theorem~\ref{thm:prop RAAGs}, one deduces Corollary~\ref{conj:univ}.

\begin{rem}\rm
Unlike pro-$p$ RAAGs, not every generalized pro-$p$ RAAG with associated graph a graph with the diagonal property
occurs as the maximal pro-$p$ Galois group $\calG_{\K}$ of a field $\K$ containing a root of 1 of order $p$,
as shown by the following examples.
\begin{itemize}
 \item[(a)] Let $\calG$ be as in Example~\ref{exam:prop star}.
 If $\alpha_i\neq\alpha_j$ for some $i,j$,
then $\calG\not\simeq\calG_{\K}$ for any field $\K$ containing a root of 1 of order $p$ by \cite{eq:kummer}*{Ex.~8.3}. 
\item[(b)] Let $\calG$ be as in Example~\ref{exam:propraag1}.
If $\alpha_1=\beta_4=p$ and $\beta_1=\alpha_4=0$ then $\calG\not\simeq\calG_{\K}$ for any field $\K$ containing
a root of 1 of order $p$ by \cite{QSV}*{Thm.~5.29}.
\end{itemize}
\end{rem}


\end{document}